\documentclass[reqno,10pt]{amsart}
\usepackage[english]{babel}
\usepackage{amsthm,hyphenat}
\usepackage{amsmath}
\usepackage{amssymb}
\usepackage[latin1]{inputenc}
\usepackage{dsfont}
\usepackage[bookmarksnumbered,colorlinks]{hyperref}
\usepackage{graphics}
\usepackage{enumerate}
\def\bb#1\eb{\textcolor{blue}
{#1}} %
\newcommand{\R}{\mathbb R}

\title[Corrigendum to ``Chern connection as a family of affine connections'']{Corrigendum to ``Chern connection of a  pseudo-Finsler metric as a family of affine connections''}

\author[M. A. Javaloyes]{Miguel Angel Javaloyes}
\address{Departamento de Matem\'aticas, \hfill\break\indent
Universidad de Murcia, \hfill\break\indent
Campus de Espinardo,\hfill\break\indent
30100 Espinardo, Murcia, Spain}
\email{majava@um.es}

\date{30.04.2014}

\thanks{This work was partially supported by MINECO (Ministerio de Econom\'{\i}a y Competitividad) and FEDER (Fondo Europeo de Desarrollo Regional) project MTM2012-34037 and Fundaci\'{o}n S\'{e}neca project 04540/GERM/06, Spain. This research is a
result of the activity developed within the framework of the Programme in Support of Excellence Groups of the Regi\'{o}n de Murcia, Spain, by Fundaci\'{o}n S\'{e}neca, Regional Agency for Science and Technology (Regional Plan for
Science and Technology 2007-2010).}


\begin{document}
\newtheorem{thm}{Theorem}[section]
\newtheorem{prop}[thm]{Proposition}
\newtheorem{lemma}[thm]{Lemma}
\newtheorem{cor}[thm]{Corollary}
\theoremstyle{definition}
\newtheorem{defi}[thm]{Definition}
\newtheorem{notation}[thm]{Notation}
\newtheorem{exe}[thm]{Example}
\newtheorem{conj}[thm]{Conjecture}
\newtheorem{prob}[thm]{Problem}
\newtheorem{rem}[thm]{Remark}

\begin{abstract}
In this note we give the correct statements of \cite[Proposition 3.3 and Theorem 3.4]{Jav14} and a formula of the Chern curvature in terms of the curvature tensor $R^V$ of the affine connection $\nabla^V$ and the Chern tensor $P$.
\end{abstract}

\maketitle
\begin{section}{Curvature of two parameter maps}
Throughout this note, we will use the same notation and conventions as in \cite{Jav14} with a small exception: in local calculation $(\Omega,(u^i)^n_{i=1})$ denotes a chart on the base manifold $M$, and $(\pi^{-1}(\Omega),(x^i)^n_{i=1},(y^i)^n_{i=1})$ is the induced chart on $TM$. Then $x^i=u^i\circ\pi$, $y^i(v)=v(u^i)$ for every $v\in T_pM$, $p\in\Omega$. We agree to abbreviate composite mappings like $f\circ g$ as $f(g)$. Vector fields on $\Omega$ can naturally be regarded as local sections of the pull-back bundle $\pi^*(TM)$; we use such harmless identifications all the time.

Proposition 3.3 and Theorem 3.4 in \cite{Jav14} are not correct. In particular, the problem in Proposition 3.3 is that $R^{V}(V,U)W$ depends also on $D_\gamma^{\dot\gamma}U(\gamma(a))$, so it is not independent of the extension of $u$. The corrected versions of such results are given in Theorems \ref{Prop3.3} and \ref{curvaturaway}, respectively. In order to formulate the correct results we need to associate a curvature operator to every two-parameter map.

Let us begin with some definitions. Given a pseudo-Finsler manifold $(M,L)$ and an $L$-admissible vector field $V$ on $\Omega\subset M$, we define a $(1,3)$ tensor $P_V$ by
\[P_V(X,Y,Z)=\frac{\partial}{\partial t}\left(\nabla^{V+tZ}_XY\right)|_{t=0},\]
where  $X,Y,Z$ are arbitrary smooth vector fields on $\Omega$. Let us observe that the tensor $P_V$ is symmetric in its first two arguments because  $\nabla^V$ is torsion-free.
Moreover,  in coordinates $P_V$ is given by
\[P_V(X,Y,Z)=X^j Y^k Z^l P^i_{\,\,jkl}(V) \frac{\partial}{\partial u^k},\]
where $P^i_{\,\,jkl}(V)=\frac{\partial \Gamma^i_{\,\, jk}}{\partial y^l}(V)$. In particular, it is clear  that the value of $P_V(X,Y,Z)$ at a point $p\in \Omega$ depends only on $V(p)$ and not on the extension $V$ used to compute it. From the homogeneity of $\nabla^V$, namely, from the property $\nabla^{\lambda V}=\nabla^V$ for every $\lambda>0$, it follows easily that
\begin{equation}\label{Pprop}
P_V(X,Y,V)=0.
\end{equation}
In \cite[Eq. (7.23)]{Sh01}, this tensor is called the {\it Chern curvature}. To avoid confusion with the Chern curvature in \cite[Eq. (15)]{Jav14} we will refer to it as {\it Chern tensor}. Observe that there is a misprint in \cite[Eq. (15)]{Jav14}. The right formula for Chern curvature is
\begin{equation*}\label{curvchern}
R_v(V(q),U(q))W(q)=V^k U^l W^j(q) R_{j\,\,kl}^{\,\,\,i}(v) \frac{\partial}{\partial u^i}(q),\quad q:=\pi(v).
\end{equation*}
Let us also define the curvature of any $L$-admissible two parameter map
\[
  \Lambda: [a,b]\times (-\varepsilon,\varepsilon)\rightarrow M,\quad   (t,s)\rightarrow \Lambda(t,s).
\]
(Here `$L$-admissible' means that $\Lambda_t(t,s)\in A$ for every $(t,s)\in [a,b]\times (-\varepsilon,\varepsilon)$). To avoid problems with differentiability, we assume that this map can be extended to a smooth map defined in an open subset $(\bar{a},\bar{b})\times (-\varepsilon,\varepsilon)\subset \R^2$, with $[a,b]\subset (\bar{a},\bar{b})$. Then for every smooth vector field $W$ along $\Lambda$ we define the curvature operator
\[R^\Lambda(W):=D_{\gamma_{s}}^{\Lambda_t}D_{\beta_{t}}^{\Lambda_t}W-D_{\beta_{t}}^{\Lambda_t}D_{\gamma_{s}}^{\Lambda_t}W,\]
(see the notation in \cite[Section 3.1]{Jav14}). In the following theorem we will relate the curvature of a two parameter map with the Chern curvature and the tensor $P_V$.
Then we obtain the correct version of \cite[Theorem 3.4]{Jav14}.
\begin{thm}\label{curvaturaway}
Let $(M,L)$ be a pseudo-Finsler manifold. Consider an $L$-ad\-mis\-sible smooth curve $\gamma:[a,b]\rightarrow M$,  an $L$-admissible two parameter map $\Lambda: [a,b]\times (-\varepsilon,\varepsilon)\rightarrow M$ such that $\Lambda(\cdot,0)=\gamma$ and  a smooth vector field $W$ along $\Lambda$. With the above notation,
\begin{equation}\label{curformula}
R^\Lambda(W)=R_{\dot\gamma}(\dot\gamma,U)W+P_{\dot\gamma}(U,W,D_\gamma^{\dot\gamma}\dot\gamma)-P_{\dot\gamma}(\dot\gamma,W,D_\gamma^{\dot\gamma}U),
\end{equation}
where $U$ is the variational vector field of $\Lambda$ along $\gamma$, namely, $U(t)=\Lambda_s(t,0)$.
\end{thm}
\begin{proof}
Observe that using the formula for $D_{\gamma_{s}}^{\Lambda_t}D_{\beta_{t}}^{\Lambda_t}W-D_{\beta_{t}}^{\Lambda_t}D_{\gamma_{s}}^{\Lambda_t}W$ in the proof of \cite[Proposition 3.3]{Jav14} and formulas (13) and (14) in \cite{Jav14}, we get
\begin{multline*}
 R^\Lambda(W)=
 \left[W^iU^j\dot\gamma^p \frac{\partial\Gamma_{\,\,ij}^k}{\partial x^p}
 (\dot\gamma)+W^iU^j\Lambda_{tt}^p\frac{\partial\Gamma_{\,\,ij}^k}{\partial y^p}
 (\dot\gamma)\right.\\-W^i\dot\gamma^j U^p\frac{\partial \Gamma_{\,\,ij}^k}{\partial x^p}(\dot\gamma)-W^i\dot\gamma^j \Lambda_{ts}^p\frac{\partial \Gamma_{\,\,ij}^k}{\partial y^p}(\dot\gamma)\\ \left.
 +W^iU^j \dot\gamma^m \left(\Gamma_{\,\,ij}^l(\dot\gamma) \Gamma_{\,\,lm}^k(\dot\gamma)-\Gamma_{\,\,im}^l(\dot\gamma) \Gamma_{\,\,lj}^k(\dot\gamma)\right)\right]\frac{\partial}{\partial u^k}(\gamma).
 \end{multline*}
Since
 \begin{align*}
\Lambda^p_{tt}&=(D_\gamma^{\dot\gamma}\dot\gamma)^p-\dot\gamma^i\dot\gamma^j \Gamma^p_{\,\,ij}(\dot\gamma),\\
\Lambda^p_{ts}&=(D_\gamma^{\dot\gamma}U)^p-\dot\gamma^iU^j \Gamma^p_{\,\,ij}(\dot\gamma),
\end{align*}
we find that
\begin{multline}\label{Rcur}
 R^\Lambda(W)=
 \left[W^iU^j\dot\gamma^p \frac{\partial\Gamma_{\,\,ij}^k}{\partial x^p}
 (\dot\gamma)-W^iU^j\dot\gamma^m\dot\gamma^n \Gamma_{\,\,mn}^p(\dot\gamma)\frac{\partial\Gamma_{\,\,ij}^k}{\partial y^p}
 (\dot\gamma)\right.\\-W^i\dot\gamma^j U^p\frac{\partial \Gamma_{\,\,ij}^k}{\partial x^p}(\dot\gamma)+W^i\dot\gamma^j \dot\gamma^m U^n \Gamma_{\,\,mn}^p(\dot\gamma)\frac{\partial \Gamma_{\,\,ij}^k}{\partial x^p}(\dot\gamma)\\
 +W^iU^j(D_\gamma^{\dot\gamma}\dot\gamma)^p\frac{\partial\Gamma_{\,\,ij}^k}{\partial y^p}-
 W^i\dot\gamma^j (D_\gamma^{\dot\gamma}U)^p\frac{\partial \Gamma_{\,\,ij}^k}{\partial y^p}(\dot\gamma)\\\left.
 +W^iU^j\dot\gamma^m  \left(\Gamma_{\,\,ij}^l(\dot\gamma) \Gamma_{\,\,lm}^k(\dot\gamma)-\Gamma_{\,\,im}^l(\dot\gamma) \Gamma_{\,\,lj}^k(\dot\gamma)\right)\right]\frac{\partial}{\partial u^k}(\gamma).
 \end{multline}
Finally, \eqref{curformula2} follows easily from the last relation and definitions taking into account that $\dot\gamma^i\Gamma^k_{\,\, ij}(\dot\gamma)=N^k_{\,\, j}(\dot\gamma)$.
\end{proof}
As a consequence of Theorem \ref{curvaturaway},   we can define
\[R^\gamma(\dot\gamma,U)W:=R^\Lambda(\tilde{W})\]
for any vector fields $U$ and $W$ along $\gamma$, where $\Lambda$ is any $L$-admissible two parameter map such that $\Lambda(\cdot,0)=\gamma$ and $\Lambda_s(t,0)=U(t)$, and $\tilde{W}$ is any extension of $W$ to $\Lambda$.  The last theorem ensures that $R^\gamma$ does not depend on the choice of two parameter map, neither on the extension of $W$.
\begin{lemma}\label{nullP}
Given a pseudo-Finsler manifold $(M,L)$ and its Chern tensor $P$, we have that $P_v(v,v,u)=0$ for any $v\in A$ and $u\in T_{\pi(v)}M$.
\end{lemma}
\begin{proof}
It is enough to show that $y^iy^j\frac{\partial \Gamma_{\,\,ij}^k}{\partial y^p}=0$ for any $k=1,\ldots,n$.
Using that $y^i\Gamma^k_{\,\, ij}=N^k_{\,\, j}$ we get
\begin{equation}\label{interm}
  y^i \frac{\partial \Gamma^k_{\,\, ij}}{\partial y^l}=\frac{\partial N^k_{\,\,j}}{\partial y^l}-\Gamma^k_{\,\, lj}.
\end{equation}
Since the functions $N^i_{\,\, j}(v)$ are positive homogeneous of degree one, we have
\begin{equation}\label{relta}
  y^l\frac{\partial N^k_{\,\,j}}{\partial y^l}=N^k_{\,\,j}.
\end{equation}
Using  \eqref{interm} and $y^j\Gamma^k_{\,\, lj}=N^k_{\,\, l}$, it follows that
\begin{equation}\label{interm2}
y^j y^i \frac{\partial \Gamma^k_{\,\, ij}}{\partial y^l}=y^j\frac{\partial N^k_{\,\,j}}{\partial y^l}-y^j\Gamma^k_{\,\, lj}=y^j\frac{\partial N^k_{\,\,j}}{\partial y^l}-N^k_{\,\,l}.
\end{equation}
Introduce the spray coefficients $G^i=\gamma^i_{\,\, jk} y^jy^k$ and observe that $\frac 12 \frac{\partial G^i}{\partial y^j}=N^i_{\,\,j}$ (see \cite[Eq. (3.8.2)]{BaChSh00}). Using the last relation and \eqref{relta}, we get
\begin{equation*}
y^j\frac{\partial N^k_{\,\,j}}{\partial y^l}=y^j\frac{1}{2} \frac{\partial^2 G^k}{\partial y^l\partial y^j}= y^j \frac{\partial N^k_{\,\,l}}{\partial y^j}= N^k_{\,\, l}.
\end{equation*}
Substituting this in \eqref{interm2} we finally conclude that $y^j y^i \frac{\partial \Gamma^k_{\,\, ij}}{\partial y^l}=0$.
\end{proof}
\begin{cor}
For any $L$-admissible curve $\gamma:[a,b]\rightarrow M$ we have
\begin{equation}\label{curformula2}
R^\gamma(\dot\gamma,U)\dot\gamma=R_{\dot\gamma(a)}(\dot\gamma(a),U)\dot\gamma+P_{\dot\gamma}(U,\dot\gamma,D_\gamma^{\dot\gamma}\dot\gamma).
\end{equation}
Therefore, the value of $R^\gamma(\dot\gamma,U)\dot\gamma$ at $t=s_0\in [a,b]$ depends only on $U(s_0)$ and $\gamma$, and not on the particular extension of $U$ used to compute it.
\end{cor}
\begin{proof}
A straightforward consequence from \eqref{curformula} and Lemma \ref{nullP}.
\end{proof}

\section{Relation with the curvature tensor $R^V$}
Let us see  how the curvature tensor $R^V$ relates to the Chern curvature $R_V$ and the Chern tensor $P_V$.
\begin{thm}
Let $(M,L)$ be a pseudo-Finsler manifold, $V$ an $L$-admissible vector field on an open subset $\Omega\subset M$, and let $X,Y,Z$ be arbitrary smooth vector fields on $\Omega$. Then
\begin{equation}\label{cherncur}
R^V(X,Y)Z=R_V(X,Y)Z+P_V(Y,Z,\nabla^V_XV)-P_V(X,Z,\nabla^V_YV).
\end{equation}
\end{thm}
\begin{proof}
By definition, using the same notation as in the proof of \cite[Proposition 3.3]{Jav14}, over $\Omega$ we have
\begin{multline}
 R^V(X,Y)Z=
 \left[Z^iY^jX^p \frac{\partial\tilde{\Gamma}_{\,\,ij}^k}{\partial u^p}-Z^iX^j Y^p\frac{\partial \tilde{\Gamma}_{\,\,ij}^k}{\partial u^p}\right.\\\left.
 +Z^iY^j X^m \left(\tilde{\Gamma}_{\,\,ij}^l \tilde{\Gamma}_{\,\,lm}^k-\tilde{\Gamma}_{\,\,im}^l \tilde{\Gamma}_{\,\,lj}^k\right)\right]\frac{\partial}{\partial u^k}.\label{Firstcur}
 \end{multline}
 Now observe that $\frac{\partial\tilde{\Gamma}_{\,\,ij}^k}{\partial u^p}
 =\frac{\partial\Gamma_{\,\,ij}^k}{\partial x^p}(V)+\frac{\partial V^l}{\partial u^p}\frac{\partial\Gamma_{\,\,ij}^k}{\partial y^l}(V)$ and $(\nabla^V_XV)^k=X^p\frac{\partial V^k}{\partial u^p}+X^pV^l\Gamma_{\,\,pl}^k(V)$. Using this, we conclude that
 \begin{multline*}
 X^p \frac{\partial\tilde{\Gamma}_{\,\,ij}^k}{\partial u^p}
 =X^p \frac{\partial\Gamma_{\,\,ij}^k}{\partial x^p}(V)+X^p\frac{\partial V^l}{\partial u^p}\frac{\partial\Gamma_{\,\,ij}^k}{\partial y^l}(V)\\
 =X^p \frac{\partial\Gamma_{\,\,ij}^k}{\partial x^p}(V)+
 (\nabla^V_XV)^l\frac{\partial\Gamma_{\,\,ij}^k}{\partial y^l}(V)
 -X^mV^n\Gamma_{\,\,mn}^l(V)\frac{\partial\Gamma_{\,\,ij}^k}{\partial y^l}(V).
 \end{multline*}
Taking into account  that $V^n\Gamma_{\,\,mn}^l(V)=N^l_{\,\,m}(V)$, we get
 \begin{equation}\label{eq1}
 X^p \frac{\partial\tilde{\Gamma}_{\,\,ij}^k}{\partial u^p}
 =X^p \frac{\partial\Gamma_{\,\,ij}^k}{\partial x^p}(V)+
 (\nabla^V_XV)^l\frac{\partial\Gamma_{\,\,ij}^k}{\partial y^l}(V)
 -X^mN^l_{\,\,m}(V)\frac{\partial\Gamma_{\,\,ij}^k}{\partial y^l}(V).
 \end{equation}
 In the same way,
 \begin{equation}\label{eq2}
 Y^p \frac{\partial\tilde{\Gamma}_{\,\,ij}^k}{\partial u^p}
 =Y^p \frac{\partial\Gamma_{\,\,ij}^k}{\partial x^p}(V)+
 (\nabla^V_YV)^l\frac{\partial\Gamma_{\,\,ij}^k}{\partial y^l}(V)
 -Y^mN^l_{\,\,m}(V)\frac{\partial\Gamma_{\,\,ij}^k}{\partial y^l}(V).
 \end{equation}
 Substituting \eqref{eq1} and \eqref{eq2} in \eqref{Firstcur}, we obtain \eqref{cherncur} in coordinates.
\end{proof}
Finally we can give the correct version of \cite[Propositon 3.3]{Jav14}.
\begin{thm}\label{Prop3.3}
Let $\gamma:[a,b]\rightarrow M$ be a smooth embedded $L$-admissible curve and $V$ an $L$-admissible smooth vector field defined on an open subset $\Omega\subset M$. Assume that $\gamma([a,b])\subset \Omega$ and $V$ coincides with $\dot\gamma$ along $\gamma$. Then
\begin{equation}\label{finaloutput}
R^\gamma(\dot\gamma,U)W=\left(R^V(V,\tilde U)\tilde W+P_V(V,\tilde W,[\tilde{U},V])\right)(\gamma),
\end{equation}
where $U$ and $W$ are smooth vector fields along $\gamma$, and $\tilde{U},\tilde W$ are extension of $U,W$ to $\Omega$, resp.
\end{thm}
\begin{proof}
This follows easily from \eqref{curformula} and \eqref{cherncur} since $\nabla^V$ is torsion-free.
\end{proof}
\begin{rem}
Observe that the expression in \cite[Corollary 3.5]{Jav14} is valid. Indeed, more generally, for every $v\in A$ and $u,w\in T_{\pi(v)}M$, it holds that
\begin{equation}\label{predecessor}
K_v(u,w)=\frac{g_v((R^{\gamma_v}(\dot\gamma_v,U)W)(t_0),\dot\gamma_v(t_0))}{L(v)g_v(u,w)-g_v(v,u)g_v(v,w)},
\end{equation}
where $\gamma_v$ is the geodesic such that $\dot\gamma_v(t_0)=v$ and $U,W$ are arbitrary extensions of $u,w$ along $\gamma_v$. Recall the $K_v(u,w)$, the predecessor of the flag curvature, is defined as
\[
K_v(u,w)=\frac{g_v(R_{v}(v,u)w,v)}{L(v)g_v(u,w)-g_v(v,u)g_v(v,w)}.\]
To prove \eqref{predecessor}, we show that if $\gamma$ is a geodesic, then
\begin{equation}\label{FlagCur}
g_{\dot\gamma}(R^\gamma(\dot\gamma,U)W,\dot\gamma)=-g_{\dot\gamma}(R^\gamma(\dot\gamma,U)\dot\gamma,W)
\end{equation}
where $g_{\dot\gamma}$ is given by the rule $g_{\dot\gamma}(X,Y):=g_{\dot\gamma(t)}(X(t),Y(t))$ for any two vector fields $X,Y$ along $\gamma$.
This holds trivially in the interior of the set where $U$ is proportional to $\dot\gamma$, because in this case $[\tilde{U},V]$ is proportional to $V$, and then  applying \eqref{Pprop} and the antisymmetry of $R^V$ in its first two arguments to \eqref{finaloutput}, we get
$R^\gamma(\dot\gamma,U)W=R^\gamma(\dot\gamma,U)\dot\gamma=0$. If $\dot\gamma$ and $U$ are linearly independent, then we can choose extensions $V$ and $\tilde{U}$ with $[\tilde{U},V]=0$, and \eqref{finaloutput} together with \cite[Proposition 3.1]{Jav14} and \cite[Lemma 3.10]{JavSoa14} conclude \eqref{FlagCur}. By continuity we can extend \eqref{FlagCur} to the interval of definition of $\gamma$. As the right hand side of \eqref{FlagCur} does not depend on the extension $U$ of $u$ along $\gamma$, we can compute the left hand side assuming that $D^{\dot\gamma}_\gamma U=0$, and using \eqref{curformula}, we get \eqref{predecessor}.
\end{rem}
\end{section}
\section*{Acknowledgments}
I would like to warmly acknowledge D. Kertesz and professor J. Szilasi  for pointing out the mistake in \cite[Proposition 3.3]{Jav14}.

\end{document}